\documentclass[11pt,times,twoside]{article}

\usepackage{graphicx,latexsym,euscript,makeidx,color}
\usepackage{amsmath,amsfonts,amssymb,amsthm,mathrsfs}
\usepackage{enumerate}

\usepackage[colorlinks,linkcolor=blue,anchorcolor=green,citecolor=red]{hyperref}%when print, use the next package
\usepackage{geometry}
\def\5n{\negthinspace \negthinspace \negthinspace \negthinspace \negthinspace }
\def\4n{\negthinspace \negthinspace \negthinspace \negthinspace }
\def\3n{\negthinspace \negthinspace \negthinspace }
\def\2n{\negthinspace \negthinspace }
\def\1n{\negthinspace }
      \def\ba{\bar{a}}
   \def\cB{{\cal B}}

\def\dbE{\mathbb{E}}     
\def\dbF{\mathbb{F}} \def\sF{\mathscr{F}}      
     
\def\dbH{\mathbb{H}}

\def\dbP{\mathbb{P}}     
\def\dbQ{\mathbb{Q}}     
\def\dbR{\mathbb{R}}     
     
 \def\sT{\mathscr{T}}

     \def\bY{\bar{Y}} \def\hY{{\hat Y}}  \def\by{\bar{y}} \def\hy{{\hat y}}
   \def\cZ{{\cal Z}}  \def\bZ{\bar{Z}} \def\hZ{{\hat Z}}  \def\bz{\bar{z}} \def\hz{{\hat z}}
\def\Om{\Omega}           
\def\ss{\smallskip}                
\def\ms{\medskip}                
               
\def\ds{\displaystyle}

\def\no{\noindent}        \def\q{\quad}                      
\def\ns{\noalign{\ss}}    \def\qq{\qquad}                    
    \def\hb{\hbox}                     
                   
         \def\rf{\eqref}                    
  \def\deq{\triangleq}               
\def\tb{\textcolor{blue}}            \def\({\Big (}
\def\les{\leqslant}                  \def\){\Big )}
\def\ges{\geqslant}       \def\esssup{\mathop{\rm esssup}}   \def\[{\Big[}
           \def\]{\Big]}
                   \def\cd{\cdot}

        \def\ts{\times}  \def\leq{\leqslant}
        \def\G{\Gamma}         \def\o{\omega}
\def\b{\beta}            \def\d{\delta}        
           \def\th{\theta}    
\def\e{\varepsilon}   \def\L{\Lambda}  \def\l{\lambda}        
    \def\t{\tau}     \def\f{\varphi}  \def\i{\infty}   
\def\cd{\cdot}        
\def\bde{\begin{definition}\label}    \def\ede{\end{definition}}
\def\be{\begin{equation}}
\def\bel{\begin{equation}\label}      \def\ee{\end{equation}}
\def\bt{\begin{theorem}\label}        \def\et{\end{theorem}}
\def\bc{\begin{corollary}\label}      \def\ec{\end{corollary}}
\def\bl{\begin{lemma}\label}          \def\el{\end{lemma}}
\def\bp{\begin{proposition}\label}    \def\ep{\end{proposition}}
\def\bas{\begin{assumption}\label}    \def\eas{\end{assumption}}
\def\br{\begin{remark}\label}         \def\er{\end{remark}}
\def\bex{\begin{example}\label}       \def\ex{\end{example}}
\def\ba{\begin{array}}                \def\ea{\end{array}}
\def\ben{\begin{enumerate}}           \def\een{\end{enumerate}}

\newtheorem{theorem}{Theorem}[section]
\newtheorem{definition}[theorem]{Definition}
\newtheorem{proposition}[theorem]{Proposition}
\newtheorem{corollary}[theorem]{Corollary}
\newtheorem{lemma}[theorem]{Lemma}
\newtheorem{assumption}[theorem]{Assumption}
\newtheorem{remark}[theorem]{Remark}
\newtheorem{example}[theorem]{Example}

\makeatletter
   
   \@addtoreset{equation}{section}
\makeatother

\begin{document}

\title{\bf Solvability of a class of mean-field BSDEs with quadratic growth}

\author{
Tao Hao\thanks{School of Statistics, Shandong University of Finance and Economics, Jinan 250014, China (Email: {\tt taohao@sdufe.edu.cn}).
This author is partially supported by National Natural Science Foundation of China
(Grant No. 11871037), Natural Science Foundation of Shandong Province (Grant No. ZR2020MA032),
 and the Colleges and Universities Youth Innovation Technology Program of Shandong Province (Grant No. 2019KJI011).}~,~~~
Jiaqiang Wen\thanks{Corresponding author. Department of Mathematics and SUSTech International Center for Mathematics, Southern University of Science and Technology,
Shenzhen, Guangdong, 518055, China (Email: {\tt wenjq@sustech.edu.cn}).
This author is supported by National Natural Science Foundation of China (grant No. 12101291) and
Guangdong Basic and Applied Basic Research Foundation (grant No. 2022A1515012017).}~,~~~
Jie Xiong\thanks{Department of Mathematics and SUSTech International center for Mathematics, Southern University of Science and Technology, Shenzhen 518055, Guangdong, China {\tt (xiongj@sustech.edu.cn)}. This author is supported by Southern University of Science and Technology Start up fund Y01286120 and the National Natural Science Foundation of China (Grants Nos. 61873325, 11831010).} }

\date{}
\maketitle

\no\bf Abstract: \rm
In this paper, we study the multi-dimensional mean-field backward stochastic differential equations (BSDEs, for short) with quadratic growth.
Under small terminal value, the existence and uniqueness are proved for the multi-dimensional situation when the generator $f(t,Y,\dbE[Y],Z,\dbE[Z])$ is of quadratic growth with respect to the last four items, using some new methods. Besides, a kind of comparison theorem is obtained.

\ms

\no\bf Key words: \rm mean-field backward stochastic differential equation, backward stochastic
differential equation, quadratic growth.

\ms

\no\bf AMS subject classifications. \rm 60H10, 60H30

\section{Introduction}

Mean-field stochastic differential equations (SDEs, for short), also called McKean-Vlasov equations, can be traced back to the works of Kac \cite{Kac-1956} in the 1950s.
Recently, mean-field backward stochastic differential equations (BSDEs, for short) were introduced by  Buckdahn, Djehiche, Li and Peng \cite{BDLS-09-AP} and Buckdahn, Li and Peng \cite{Buckdahn2}, owing to that mathematical mean-field approaches play an important role in many fields, such as economics, physics, and game theory (see Lasry and Lions \cite{Lasry}). Since then, mean-field BSDEs have received intensive attention.
In order to present the work more clearly, we describe the problem in detail.

\ms

Assume that $\{W_t;0\les t<\i\}$ is a $d$-dimensional standard Brownian motion defined on a complete filtered probability space $(\Om,\sF,\dbF,\dbP)$, where $\dbF=\{\sF_t; 0\les t<\i\}$ is the filtration generated by $W$ and augmented by all the $\dbP$-null sets in $\sF$.
Consider the following mean-field backward stochastic differential equations over a finite horizon $[0,T]$:
\bel{BSDE}\left\{\2n\ba{ll}
\ds -dY_t =f(t,Y_t,\dbE[Y_{t}],Z_t,\dbE[Z_{t}])dt-Z_tdW_t,\q~t\in[0,T];\\
\ns\ns\ds Y_T=\xi,\ea\right.\ee
where the random variable $\xi$ is called the {\it terminal value} and the mapping $f(\cd)$ is called the {\it generator}. The pair of unknown processes $(Y,Z)$, called an adapted solution of $(\ref{BSDE})$, are  $\dbF$-adapted with values in $\dbR^m\ts\dbR^{m\ts d}$.
For convenience, hereafter, by {\it quadratic} mean-field BSDEs, or mean-field BSDEs with {\it quadratic growth}, we mean that in  \rf{BSDE}, the generator $f(\cd)$ grows in $Z$ quadratically.
Meanwhile, we call $\xi$ the {\it bounded terminal value} if it is bounded. In addition,
we call $\xi$ the {\it small terminal value}, if there exists a small positive constant $\e$ such that $\|\xi\|_\i\les \e$.

\ms

As a natural extension of BSDEs (see below for precise description),
the mean-field BSDEs \rf{BSDE} were introduced by Buckdahn et al. \cite{BDLS-09-AP,Buckdahn2}, where they established the existence and uniqueness of adapted solutions under the condition that $f(\cd)$ is uniformly Lipschitz in the last four arguments.
From then on, the theory and applications of mean-field BSDEs have been developed significantly.
For example, Carmona and Delarue  provided a detailed probabilistic analysis of the optimal control of nonlinear stochastic dynamical systems of the McKean-Vlasov type in \cite{Carmona-Delarue-AP2015}, and
studied some special class of quadratic forward-backward stochastic differential equations (FBSDEs, for short) of mean-field type in \cite{CD18}.
Buckdahn et al. \cite{Buckdahn2017} studied the general mean-field stochastic differential equations and their relation with the associated PDEs.
Zhang, Sun and Xiong \cite{Zhang-Sun-Xiong-SICON2018} obtained a general stochastic maximum principle for a Markov regime switching jump-diffusion model of mean-field type.
Briand,  Elie and Hu \cite{Hu-AAP2018} introduced and studied the BSDEs with mean reflection.
Li, Sun and Xiong \cite{Li-Sun-Xiong-AMO2019} established the results for linear quadratic optimal control problems for mean-field backward stochastic differential equations.
Douissi, Wen and Shi \cite{Douissi-Wen-Shi 2019} and Shi, Wen and Xiong \cite{Shi-Wen-Xiong 2021} studied the optimal control problem of mean-field forward-backward stochastic systems driven  by fractional Brownian motion.

\ms

Now, we recall the following nonlinear BSDEs:
\bel{BSDE1}\left\{\2n\ba{ll}
\ds -dY_t =f(t,Y_t,Z_t)dt-Z_tdW_t,\q~t\in[0,T];\\
\ns\ns\ds Y_T=\xi.\ea\right.\ee
When the generator $f(\cd)$ is of linear growth with respect to $(Y,Z)$, the existence and uniqueness of \rf{BSDE1} were firstly proved by Pardoux and Peng \cite{Pardoux-Peng-90}. Since then, a lot of researchers have found that BSDEs have important applications in mathematical finance, stochastic optimal control and partial differential equation (see \cite{Karoui-Peng-Quenez-97,Pardoux-Peng-92,Yong-Zhou-99}, to mention a few).
Meanwhile, owing to many important applications, a lot of efforts have been made to relax the conditions on the generator $f(\cd)$ of \rf{BSDE1} with respect to $Z$. For example,
Kobylanski \cite{Kobylanski00} proved the existence and uniqueness of one-dimensional BSDE with bounded terminal condition and with $f(\cd)$ growing quadratically in $Z$.
The well-posedness of one-dimensional quadratic BSDE with unbounded terminal value was obtained by
Briand and Hu \cite{Briand-Hu-06,Briand-Hu-08}.
Hibon et al. \cite{HHLLW-MCRF2018} studied a class of quadratic BSDEs with mean reflection. Hu, Li and Wen \cite{Hu-Li-Wen-JDE2021} obtained the existence and uniqueness of anticipated BSDEs with quadratic growth.
Some other recent developments of quadratic BSDEs can be found in Bahlali, Eddahbi and Ouknine \cite{Bahali-Eddahbi-Ouknine-17}, Barrieu and El Karoui \cite{Briand-Karoui-13}, Cheridito and Nam \cite{Cheridito-Nam-14}, Hibon, Hu and Tang \cite{Hibon-Hu-Tang17}, Hu and Tang \cite{Hu-Tang-16}, Tevzadze \cite{Tevzadze-08},
 Xing and Zitkovic \cite{Xing-Zitkovic}, and references cited therein.

\ms

Let us introduce a motivation to study (\ref{BSDE}) at the particles level. Consider the following particle system with $N$ particles:
$$ Y^i_t=\xi^i
+\int_t^Tf^i\Big(s,Y^i_s,\frac{1}{N}\sum\limits_{i=1}^NY^i_s,Z^{ii}_s,\frac{1}{N}\sum\limits_{i=1}^NZ^{ii}_s\Big)ds
-\sum\limits_{j=1}^N\int_t^TZ^{ij}_sdW^j_s,\q~ t\in[0,T], $$
where for each $i,j=1,...,N$, $Z^{ij}$ is an $m\times d$ matrix;
$\{W^j;1\leq j\leq N\}$ are $N$ independent $d$-dimensional Brownian motions;
$\xi^i$ and $f^i$ are $N$ independent copies of $\xi$ and $f$, respectively.
Following Lions's idea and the law of large numbers, we conjecture that when $N$ tends to $\infty$,  the mean-field limit of solutions of the above particle system corresponds to that of the mean-field BSDE \rf{BSDE}.

\ms

On the other hand, in the last two decades, stimulated by the broad applications and the open problem proposed by Peng \cite{Peng-98}, a lot of efforts have been made to relax the conditions on the generator $f$ of the mean-field BSDEs.
Hibon, Hu and Tang \cite{Hibon-Hu-Tang17} considered quadratic mean-field BSDEs in one-dimensional situation when the generator $f$ depends on the expectation of $(Y,Z)$, and studied the existence and uniqueness of related equations.
However, the mean-field BSDEs \rf{BSDE} with quadratic growth in the multi-dimensional situation is still open.
In this paper, along with the work of \cite{Hibon-Hu-Tang17},
we study the solvability of mean-field BSDEs \rf{BSDE} with quadratic growth in the multi-dimensional situation.
First, borrowing some ideas from Tevzadze \cite{Tevzadze-08}, we construct an artful method to prove that, under small terminal value, the mean-field BSDE \rf{BSDE} admits a unique adapted solution in the multi-dimensional situation. It should be pointed out that, besides the multidimensional situation, our contribution also includes the generator $f(\cd)$ is of quadratic growth with respect to all items of $(Y,\dbE[Y],Z,\dbE[Z])$.
Then, a comparison theorem for such equations is obtained for the one-dimensional situation.
%%It should be pointed out that this work is a follow us of Hibon, Hu and Tang \cite{Hibon-Hu-Tang17}.

\ms

This article is organized as follows.
Some preliminaries are presented in Section 2.
The existence and uniqueness of multi-dimensional quadratic mean-field BSDEs \rf{BSDE}  with small terminal value are proved by the fixed point argument in Section 3. A comparison theorem for such BSDEs is given in Section 4.

\section{Preliminaries}

Recall that $\{W_t\;;0\les t<\i\}$ is a $d$-dimensional standard Brownian motion defined on
the complete filtered probability space $(\Om,\sF,\dbF,\dbP)$, where $\dbF=\{\sF_t\}_{t\ges0}$ is the filtration generate generated by $W$ and augmented by all the $\dbP$-null sets in $\sF$.
We denote by $\dbR^{m\ts d}$ the space of the $m\ts d$-matrix $C$ with Euclidean norm $|C|=\sqrt{tr(CC^*)}$.
In the following, for Euclidean space $\dbH$ and $t\in[0,T]$, we denote
$$\ba{ll}
\ns\ds L^\i_{\sF_t}(\Om;\dbH)=\Big\{\th:\Om\to\dbH\bigm|\th\hb{ is $\sF_t$-measurable, }
\|\th\|_{\i}\triangleq\esssup_{\o\in\Om}|\th(\omega)|<\i\Big\},\\
\ea$$
$$\ba{ll}
\ns\ds L_\dbF^2(t,T;\dbH)\1n=\1n\Big\{\f:[t,T]\1n\times\1n\Om\to\dbH\bigm|\f\hb{ is
$\dbF$-progressively measurable, }\\
\ns\ds\qq\qq\qq\qq\qq\qq\q~
\|\f\|_{L_\dbF^2(t,T)}\deq\1n\(\dbE\int^T_t\1n|\f_s|^2ds\)^{1\over2}\1n<\2n\i\Big\},\\
\ns\ds L_\dbF^\infty(t,T;\dbH)=\Big\{\f:[t,T]\times\Om\to\dbH\bigm|\f\hb{ is
$\dbF$-progressively measurable, }\\
\ns\ds\qq\qq\qq\qq\qq\qq\q~
\|\f\|_{L_\dbF^\infty(t,T)}\deq\esssup_{(s,\o)\in[t,T]\times\Om}|\f_s(\o)|<\i\Big\},\\
\ns\ds S_\dbF^2(t,T;\dbH)=\Big\{\f:[t,T]\times\Om\to\dbH\bigm|\f\hb{ is
$\dbF$-adapted, continuous, }\\
\ns\ds\qq\qq\qq\qq\qq\qq\q~
\|\f\|_{S_\dbF^2(t,T)}\deq\Big\{\dbE\(\sup_{s\in[t,T]}|\f_s|^2\)\Big\}^{\frac{1}{2}}<\i\Big\}.\ea$$
Let $M=(M_t,\sF_t)$ be a uniformly integrable martingale with $M_0=0$, and for $p\in[1,\i)$, we set
$$\|M\|_{BMO_p(\dbP)}\deq
\sup_\t\bigg\|\dbE_\t\Big[\big(\langle M\rangle_{\t}^{\i}\big)^{\frac{p}{2}}\Big]^{\frac{1}{p}}\bigg\|_{\i},$$
where the supremum is taken over all $\dbF$-stopping times $\t$, and $\dbE_\t$ is the conditional expectation given $\sF_{\t}$. The class $\big\{M: \|M\|_{BMO_p(\dbP)}<\i\big\}$ is denoted by $BMO_p(\dbP)$. Observe that $\| \cd \|_{BMO_p}$ is a norm on this space and $BMO_p(\dbP)$ is a Banach space. In the sequel, we denote $BMO(\dbP)$ the space of $BMO_2(\dbP)$ for simplicity. Next, for any $Z\in L^2_\dbF(0,T;\dbH)$, by Burkholder-Davis-Gundy's inequalities, one has
$$\ba{ll}
\ns\ds c_2\dbE_\t\[\(\int_\t^T|Z_s|^2ds\)\]\les\dbE_\t\[\sup_{t\in[\t,T]}
\Big|\int_\t^tZ_sdW_s\Big|^2\]\les C_2\dbE_\t\[\(\int_\t^T|Z_s|^2ds\)\],\ea$$
for some constants $c_2,C_2>0$. Thus,
$$\ba{ll}
\ns\ds c_2\sup_{\t\in\sT[t,T]}\Big\|\dbE_\t\[\(\int_\t^T|Z_s|^2ds\)\]\Big\|_\i
\les\sup_{\t\in\sT[t,T]}\Big\|\dbE_\t\[\sup_{t\in[\t,T]}\Big|\int_\t^tZ_sdW_s\Big|^2\]\Big\|_\infty\\
\ns\ds\qq\qq\qq\qq\qq\qq\qq\q \les C_2\sup_{\t\in\sT[t,T]}\Big\|\dbE_\t\[\(\int_\t^T|Z_s|^2ds\)\]\Big\|_\i,\ea$$
where $\sT[t,T]$ denotes the set of all $\dbF$-stopping times $\t$ valued in $[t,T]$. Note that the above quantities could be infinite. Therefore, we introduce the following:
$$\cZ^2[t,T]=\Big\{Z\in L^2_\dbF(t,T;\dbH)\Bigm|\|Z\|_{\cZ^2[t,T]}\equiv\sup_{\t\in\sT[t,T]}\Big\|
\dbE_\t\[\int_\t^T|Z_s|^2ds\] \Big\|_\i^{1\over2}<\i\Big\}.$$
Recall that for $Z\in\cZ^2[0,T]$, the process $s\mapsto\int_0^sZ_rdW_r$ (denoted by $Z\cd W$), $s\in[0,T]$, is a {\it BMO-martingale}. Moreover, note that on $[0,T]$, $Z\cd W$ belongs to $BMO(\dbP)$ if and only if $Z\in \cZ^2[0,T]$, that is,
$$\|Z\cd W\|^2_{BMO(\dbP)}\equiv \|Z\|^2_{\cZ^2[0,T]}.$$
\begin{definition}\label{8.11.21.0} \rm
A pair of processes $(Y,Z)\in S^2_{\dbF}(0,T;\dbR^m)\times
L^2_\dbF(0,T;\dbR^{m\ts d})$ is called an {\it adapted solution} of BSDE \rf{BSDE}, if $\dbP$-almost surely, it satisfies \rf{BSDE}. Moreover, if $(Y,Z)\in L^\infty_\dbF(0,T;\dbR^m)\times\cZ^2[0,T]$, it is called a bounded adapted solution.
\end{definition}

\section{Existence and Uniqueness}

In this section, we study multi-dimensional mean-field BSDEs with quadratic growth and small terminal value.
By the theory of ordinary differential equations, we see that if the generator $f(\cd)$ is super-linear with respect to $Y$, then the equations may not have global solutions. However, we pointed out that, under small bounded value, the generator could be of quadratic growth with respect to $Y$ and $Z$ (see Hu, Li and Wen \cite{Hu-Li-Wen-JDE2021}).
Now let us introduce the following hypothesis.

\begin{assumption}\label{3.1} \rm
Let $C$ be a positive constant. For all
$s\in[0,T]$, $y_i,\bar{y}_i\in\dbR^m,z_i,\bar{z}_i\in\dbR^{m\ts d}$ with $i=1,2$,  $f(s,0,0,0,0)$ is bounded and
\begin{align*}
  \big|f(s,y_1,\bar y_1,z_1,\bar z_1)-f(s,y_2,\bar y_2,z_2,\bar z_2)\big|
\les& C\Big(|y_1|+|\bar{y}_1|+|y_2|+|\bar{y}_2|+|z_1|+|\bar{z}_1|+|z_2|+|\bar{z}_2|\Big)\\
 &\cd\Big(|y_1-y_2|+|\bar y_1-\bar y_2|+|z_1-z_2|+|\bar z_1-\bar z_2|\Big).
\end{align*}
\end{assumption}

\begin{example} \rm
Assumption \ref{3.1} implies that the generator $f(\cd)$ could be of quadratic growth with respect to the last four arguments. The following generator
$$f(s,y_1,y_2,z_1,z_2)=y^2+\bar y^2+z^2+\bar z^2,\qq s\in[0,T], y, \bar y\in\dbR^m, z, \bar z\in\dbR^{m\ts d}$$
satisfies such an assumption.
\end{example}

In the following, we state and prove the main result of this section, which establishes the existence and uniqueness of multi-dimensional BSDE \rf{BSDE} with quadratic growth and small terminal value.

\begin{theorem}\label{Theorem1}
Under the Assumption \ref{3.1}, for any bounded $\xi\in L^\i_{\sF_T}(\Om;\dbR^m)$ satisfying
\bel{20.7.3.1}\|\xi\|_{\infty}
+\bigg\|\int_0^T|f(t,0,0,0,0)|dt\bigg\|_{\i}\les \rho,\ee
where
$$\rho^2 =\frac{1}{4097C^2(T^2+1)},$$
we have that mean-field BSDE \rf{BSDE} admits a unique adapted solution $(Y,Z)$ in $\cB_\rho$, where
$$\cB_\rho\deq\Big\{(Y,Z)\in L_\dbF^\infty(0,T;\dbR^m)\times \cZ^2[0,T]\Bigm|
\|Y\|^2_{L_\dbF^\infty(0,T)}+\|Z\|_{\cZ^2[0,T]}^2\les \rho^2\Big\}.$$
\end{theorem}

\begin{proof}
The proof is divided into two steps.

\ms

\textbf{Step 1}.
We firstly consider the existence and uniqueness of the following BSDE
\begin{equation}\label{3.5.0}
  Y_{t} = \xi + \int_t^T \Big(f(s,Y_{s},\dbE[Y_{s}],Z_{s},\dbE[Z_{s}])-f(s,0,0,0,0)\Big) ds - \int_t^TZ_{s}dW_{s}, \q 0\les t\les T.
\end{equation}
In order to solve the above equation, for every $(y,z)\in L_\dbF^\infty(0,T;\dbR^m)\times \cZ^2[0,T]$,
we define the mapping $(Y,Z)=\Gamma(y,z)$ by
\begin{equation}\label{3.5}
  Y_{t} = \xi + \int_t^T \Big(f(s,y_{s},\dbE[y_{s}],z_{s},\dbE[z_{s}])-f(s,0,0,0,0)\Big) ds - \int_t^TZ_{s}dW_{s}, \q 0\les t\les T.
\end{equation}
For \rf{3.5}, using It\^{o}'s formula to $|Y|^2$ on $[t,T]$, we obtain
$$\ba{ll}
\ds|Y_t|^2+\int_t^T|Z_r|^2dr=|\xi|^2
+\int_t^T2Y_r\cdot\Big(f(r,y_{r},\dbE[y_{r}],z_{r},\dbE[z_{r}])-f(r,0,0,0,0)\Big)dr\\
\ns\ns\ds\qq\qq\qq\qq\q\ -2\int_t^TY_r\cdot Z_rdW_r.
\ea$$
Taking the conditional expectation and using the
inequality $2ab\les\frac{1}{2}a^2+2b^2$, we get
\bel{3.6}\ba{ll}
\ds |Y_t|^2+\dbE_{t}\int_t^T|Z_r|^2dr\\
\ns\ds\les \|\xi\|_{\i}^2
+2\|Y\|_{L_\dbF^\infty(0,T)}
\Bigg(\dbE_{t}\int_t^T|f(r,y_{r},\dbE[y_{r}],z_{r},\dbE[z_{r}])-f(r,0,0,0,0)|dr\Bigg)\\
\ns\ds\les\|\xi\|_{\i}^2+\frac{1}{2}\|Y\|_{L_\dbF^\infty(0,T)}^2
+2\Bigg(\dbE_{t}\int_t^T|f(r,y_{r},\dbE[y_{r}],z_{r},\dbE[z_{r}])-f(r,0,0,0,0)|dr\Bigg)^{2}.
\ea\ee
It follows from Jensen's inequality that the last term of \rf{3.6} naturally reduces to
\bel{3.7}\ba{ll}
\ds \dbE_{t}\int_t^T|f(r,y_{r},\dbE[y_{r}],z_{r},\dbE[z_{r}])-f(r,0,0,0,0)|dr\\
\ns\ds\les C\dbE_{t}\int_t^T\Big(|y_r|+|\dbE[y_{r}]|+|z_r|+|\dbE[z_{r}]|\Big)^2dr\\
\ns\ds\les 4C\dbE_{t}\int_t^T\Big(|y_r|^2+|\dbE[y_{r}]|^2+|z_r|^2+|\dbE[z_{r}]|^2\Big)dr\\
\ns\ds\les 4C\bigg[\dbE_{t}\int_t^{T}\big(|y_r|^2+|z_r|^2\big)dr+
\dbE\int_t^{T}\big(|y_r|^2+|z_r|^2\big)dr\bigg].
\ea\ee
Hence, combining \rf{3.6} and \rf{3.7}, we can obtain
$$\ba{ll}
\ds \frac{1}{2}\|Y\|^2_{L_\dbF^\infty(0,T)}+\|Z\|_{\cZ^2[0,T]}^{2}\\
\ns\ds\les
\|\xi\|_{\i}^2+32C^2\esssup_{(t,\o)\in[0,T]\times\Om}
\bigg[\dbE_{t}\int_t^{T}\big(|y_r|^2+|z_r|^2\big)dr+
\dbE\int_t^{T}\big(|y_r|^2+|z_r|^2\big)dr\bigg]^{2}\\
\ns\ds\les\|\xi\|_{\i}^2+64C^2
\bigg(T^2\|y\|_{L_\dbF^\infty(0,T)}^{4}+\|z\|_{\cZ^2[0,T]}^{4}\bigg).
\ea$$
Then, it follows from the elementary inequality $a^2+b^2\les(|a|+|b|)^2$ that
$$\ba{ll}
\ds \|Y\|^{2}_{L_\dbF^\infty(0,T)}+\|Z\|_{\cZ^2[0,T]}^{2}\\
\ns\ds \les4\|\xi\|_{\i}^2
+\b^2\Big(\|y\|_{L_\dbF^\infty(0,T)}^2+\|z\|_{\cZ^2[0,T]}^{2}\Big)^{2},
\ea$$
where $\b\deq 16C\sqrt{T^2+1}$. Now, we would like to pick $R$ such that
$$4\|\xi\|_{\i}^2+ \b^2R^4\les R^2.$$
In fact, the above inequality is solvable if and only if
\bel{19.3.21.1}\|\xi\|_{\i}\les \frac{1}{4\b}.\ee
For example, we can take
$$R=2\sqrt{2}\|\xi\|_{\i}$$
in order to satisfy this quadratic inequality. Then the ball
$$\cB_R\deq\Big\{(Y,Z)\in L_\dbF^\infty(0,T;\dbR^m)\times \cZ^2[0,T]\Bigm|
\|Y\|^2_{L_\dbF^\infty(0,T)}+\|Z\|_{\cZ^2[0,T]}^2\les R^2\Big\}$$
is such that $\G(\cB_R)\subset\cB_R$.

\ms

\textbf{Step 2}. We prove that the mapping $\G$ is a contraction on $\cB_R$.

\ms

For every $(y,z)$, $(\by,\bz)\in\cB_R$, let $(Y,Z)=\G(y,z)$ and
$(\bY,\bZ)=\G(\by,\bz)$. For simplicity of presentation, we denote
$$\hy=y-\by,\q~\hz=z-\bz,\q~\hY=Y-\bY,\q~\hY=Y-\bY.$$
Similar to the above discussion, we deduce that
$$\ba{ll}
\ds\frac{1}{2}\|\hY\|^{2}_{L_\dbF^\infty(0,T)}+\|\hZ\|_{\cZ^2[0,T]}^{2}\\
\ns\ds\les2\esssup_{(t,\o)\in[0,T]\times\Om}\bigg[\dbE_{t}\int_t^T
\Big|f(r,y_{r},\dbE[y_{r}],z_{r},\dbE[z_{r}])
-f(r,\by_{r},\dbE[\by_{r}],\bz_{r},\dbE[\bz_{r}])\Big|dr\bigg]^{2}\\
\ns\ds\les 2C^2\esssup_{(t,\o)\in[0,T]\times\Om}\bigg[\dbE_{t}\int_t^T
\Big(|y_r|+|\by_r|+\dbE[|y_{r}|]+\dbE[|\by_{r}|]+
|z_r|+|\bz_r|+\dbE[|z_{r}|]+\dbE[|\bz_{r}|]\Big)\\
\ns\ds\qq\qq\qq\qq\qq\q\
\cd\Big(|\hy_r|+|\hz_r|+\dbE\big[|\hy_{r}|]+\dbE\big[|\hz_{r}|\big]\Big)dr\bigg]^{2}\\
\ns\ds\les 2C^2\esssup_{(t,\o)\in[0,T]\times\Om}\bigg[\dbE_{t}\int_t^T
\Big(|y_r|+|\by_r|+\dbE[|y_{r}|]+\dbE[|\by_{r}|]+
|z_r|+|\bz_r|+\dbE[|z_{r}|]+\dbE[|\bz_{r}|]\Big)^2dr\\
\ns\ds\qq\qq\qq\qq~
\cd\dbE_{t}\int_t^T\Big(|\hy_r|+|\hz_r|+\dbE\big[|\hy_{r}|]+\dbE\big[|\hz_{r}|\big]\Big)^2dr\bigg]\\
\ns\ds\les 64C^2\esssup_{(t,\o)\in[0,T]\times\Om}\bigg[\dbE_{t}\int_t^T
\Big(|y_r|^2+|\by_r|^2+\dbE[|y_{r}|^2]+\dbE[|\by_{r}|^2]+
|z_r|^2+|\bz_r|^2+\dbE[|z_{r}|^2]+\dbE[|\bz_{r}|^2]\Big)dr\\
\ns\ds\qq\qq\qq\qq~
\cd\dbE_{t}\int_t^T\Big(|\hy_r|^2+|\hz_r|^2+\dbE\big[|\hy_{r}|^2]+\dbE\big[|\hz_{r}|^2\big]\Big)dr\bigg]\\
\ns\ns\ds\les
256C^2(T^2+1)
\Big[\|y\|^{2}_{L_\dbF^\infty(0,T)}+\|z\|^{2}_{\cZ^2[0,T]}+\|\by\|^{2}_{L_\dbF^\infty(0,T)}
+\|\bz\|^{2}_{\cZ^2[0,T]}\Big]\\
\ns\ns\ds\q\ \times\Big(\|\hy\|^{2}_{L_\dbF^\infty(0,T)}+\|\hz\|^{2}_{\cZ^2[0,T]}\Big).
\ea$$
Noting that
$$\|y\|^{2}_{L_\dbF^\infty(0,T)}+\|z\|^{2}_{\cZ^2[0,T]}\les R^2,\qq
\|\by\|^{2}_{L_\dbF^\infty(0,T)}+\|\bz\|^{2}_{\cZ^2[0,T]}\les R^2,$$
we obtain
$$\|\hY\|^{2}_{L_\dbF^\infty(0,T)}+\|\hZ\|_{\cZ^2[0,T]}^{2}
\les MR^2 \big(\|\hy\|^{2}_{L_\dbF^\infty(0,T)}+\|\hz\|^{2}_{\cZ^2[0,T]}\big),$$
where
$${M\deq 512C^2(T^2+1).}$$
Now, note that
${R=2\sqrt{2}\|\xi\|_{\i},}$
we have that if
\bel{19.3.21.2}{\|\xi\|_{\i}^2<\frac{1}{8M},}\ee
then
$${MR^2<1,}$$
which implies that $\G$ is a contraction on $\cB_R$. By the contraction principle, under the condition \rf{19.3.21.2}, the mapping $\G$ admits a unique fixed point, which is the solution of \rf{3.5.0}.

\ms

Finally, we come back to BSDE \rf{BSDE}, which can be rewritten into the form of \rf{3.5.0} as follows:
\begin{equation}\label{3.5.0.2}
  Y_{t} = \widetilde{\xi} + \int_t^T \Big(f(s,Y_{s},\dbE[Y_{s}],Z_{s},\dbE[Z_{s}])-f(s,0,0,0,0)\Big) ds - \int_t^TZ_{s}dW_{s}, \q 0\les t\les T,
\end{equation}
where
$$\widetilde{\xi}\deq\xi+\int_t^Tf(s,0,0,0,0)ds.$$
Note that \rf{19.3.21.2} is stronger than  \rf{19.3.21.1}. So if we define $\rho>0$ by letting
$$\rho^2\deq \frac{1}{4097C^2(T^2+1)},$$
then BSDE \rf{BSDE} admits a unique adapted solution
$(Y,Z)\in\cB_\rho$ under the following condition
\begin{equation}
\|\xi\|_{\infty}
+\bigg\|\int_0^T|f(t,0,0,0,0)|dt\bigg\|_{\i}\les \rho.
\end{equation}
This completes the proof.
\end{proof}

\begin{remark} \rm
Comparing with the results of Carmona and Delarue \cite{CD18}, where they studied some kind of quadratic FBSDE of mean-field type, we would like to show two differences.
\begin{enumerate}[~~\,\rm (i)]
\item The equations of mean-field type are different between this paper and \cite{CD18}.
In \cite{CD18}, motivated by the problem of the mean-field game,
Carmona and Delarue proved the solvability of the following FBSDE with quadratic grwoth:
\begin{equation}\label{equ 3.101}
\left\{
\begin{aligned}
\ds dX_t&=b(t,X_t, \mathcal{L}(X_t) ,\hat{\alpha}(t,X_t,\mathcal{L}(X_t),(\sigma(t,X_t,\mathcal{L}(X_t))^{-1})^\intercal Z_t))dt\\
&\q +\sigma(t,X_t,\mathcal{L}(X_t))dW_t,\\
\ns\ds dY_t&=-f(t,X_t,\mathcal{L}(X_t),\hat{\alpha}(t,X_t,\mathcal{L}(X_t),
(\sigma(t,X_t,\mathcal{L}(X_t))^{-1})^\intercal Z_t ))dt+Z_tdW_t,\\
\ns\ds X_0&=\xi\in L^2(\Omega,\mathcal{F}_0,\mathbb{P};\mathbb{R}^d),\ Y_T=g(X_T,\mathcal{L}(X_T)),
\end{aligned}
\right.
\end{equation}
where  $\mathcal{L}(X_t)$ denotes the law of the process $X_t$.
Comparing the backward equation of \rf{equ 3.101} with  \rf{BSDE}, it is easy to see that the generator $f$ of \rf{equ 3.101} depends on the law of the process $X$, however, the generator $f(\cd)$ of \rf{BSDE} depends on the expectations of $Y$ and $Z$. So the backward equation of \rf{equ 3.101} and  \rf{BSDE} are two different equations of mean-field type.
\item  The circumstances are different. The circumstance of Carmona and Delarue \cite{CD18} is Markovian, however, our model is non-Markovian. Besides, the backward equation of \rf{equ 3.101} studied in \cite{CD18} is a one-dimensional BSDE, however, \rf{BSDE} is a multi-dimensional BSDE.
\end{enumerate}
\end{remark}

\section{Comparison Theorem}

In this section, we study the comparison theorem of mean-field BSDEs with quadratic growth of the following form:
\begin{equation}\label{BSDE2}
  Y_{t} = \xi + \int_t^T f(s,Y_{s},\dbE[Y_{s}],Z_{s}) ds - \int_t^TZ_{s}dW_{s}, \q 0\les t\les T.
\end{equation}
We consider BSDE \rf{BSDE2} in one-dimensional case only, i.e., $m=1$. For simplicity of presentation, we let $d=1$ too. For this situation, we have the following lemma.
\begin{lemma}\label{3.4}
Let $C$ be a positive constant, and suppose that there are two increasing functions
$\l:\dbR^{+}\rightarrow\dbR^{+}$ and $\bar{\l}:\dbR^{+}\rightarrow\dbR^{+}$,
bounded on all bounded subsets, and a predictable process $k\in\cZ^2[0,T]$ such that
for all $s\in[0,T]$, $y,\bar{y},z\in\dbR$,
\begin{align}\label{3.1.2}
  \big|f(s,y,\bar y,z)\big|
\les& k_t^2\big[\l(|y|)+\bar\l(|\bar y|)\big]+Cz^2.
\end{align}
Then, for bounded terminal value $\xi\in L^\i_{\sF_T}(\Om;\dbR)$,
the martingale part of any bounded solution of BSDE \rf{BSDE2} belongs to the space BMO$(\dbP)$, i.e.,  $Z\in\cZ^2[0,T]$.
\end{lemma}

\begin{proof}
Let $Y$ be a solution of BSDE \rf{BSDE2} and there be a positive constant $M$ such that
$$Y_t\les M,\q\hb{ a.s. for all } t\in[0,T]. $$
So we have that $\|\xi\|_{\i}\les M$. Applying It\^{o} formula to $\exp\{\b Y_s\}$ on $[\t,T]$, we have
\begin{align*}
  \frac{\b^2}{2}\int_{\t}^Te^{\b Y_s}Z^2_sds-\b\int_{\t}^Te^{\b Y_s}f(s,Y_s,\dbE[Y_s],Z_s)ds
  +\b\int_{\t}^Te^{\b Y_s}Z_sdW_s=e^{\b\xi}-e^{\b T_{\t}}\les e^{\b M},
\end{align*}
or
\begin{align*}
  \frac{\b^2}{2}\int_{\t}^Te^{\b Y_s}Z^2_sds
  +\b\int_{\t}^Te^{\b Y_s}Z_sdW_s\les e^{\b M}+\b\int_{\t}^Te^{\b Y_s}f(s,Y_s,\dbE[Y_s],Z_s)ds,
\end{align*}
where $\b$ is a constant which will be determined later. Now, if $Z\cdot W$ is square integrable martingale, then taking conditional expectations on the above inequality, we obtain that
\begin{align*}
  \frac{\b^2}{2}\dbE_{\t}\int_{\t}^Te^{\b Y_s}Z^2_sds
  \les e^{\b M}+\b\dbE_{\t}\int_{\t}^Te^{\b Y_s}f(s,Y_s,\dbE[Y_s],Z_s)ds.
\end{align*}
Using the estimate \rf{3.1.2} we obtain that
\begin{align*}
  \frac{\b^2}{2}\dbE_{\t}\int_{\t}^Te^{\b Y_s}Z^2_sds
  \les e^{\b M}+\b[\l(M)+\bar\l(M)]\dbE_{\t}\int_{\t}^Te^{\b Y_s}k_s^2ds
  +\b C\dbE_{\t}\int_{\t}^Te^{\b Y_s}|Z_s|^2ds.
\end{align*}
Thus we have
\begin{align*}
  [\frac{\b^2}{2}-\b C] \dbE_{\t}\int_{\t}^Te^{\b Y_s}Z^2_sds
  \les e^{\b M}+\b[\l(M)+\bar\l(M)]\dbE_{\t}\int_{\t}^Te^{\b Y_s}k_s^2ds.
\end{align*}
Taking $\b=4C$, we deduce that
\begin{align*}
  4C^2\dbE_{\t}\int_{\t}^Te^{4C Y_s}Z^2_sds
  \les& e^{4C M}+4C[\l(M)+\bar\l(M)]\dbE_{\t}\int_{\t}^Te^{4C Y_s}k_s^2ds\\
  \les& e^{4C M}\Big(1+4C[\l(M)+\bar\l(M)]\|k\|_{\cZ^2[0,T]}\Big).
\end{align*}
Note that $-M\les Y\les M$, from the latter inequality we deduce that for any stopping time $\t$,
\begin{align*}
  \dbE_{\t}\int_{\t}^TZ^2_sds
  \les \frac{e^{8C M}\Big(1+4C[\l(M)+\bar\l(M)]\|k\|_{\cZ^2[0,T]}\Big)}{4C^2}.
\end{align*}
Hence $Z\cdot W\in$ BMO, i.e., $Z\in\cZ^2[0,T]$. This completes the proof.
\end{proof}

For simplicity of presentation, in the following we introduce some more notations. Let $(Y,Z)$ and $(\tilde{Y},\tilde Z)$ be two pairs of processes, and the coefficients $(f,\xi)$ and  $(\tilde f,\tilde \xi)$ be two pairs of generators. We define
\begin{align*}
   &\qq\qq\qq \d f= f-\tilde f,\qq \d \xi=\xi-\tilde \xi,\\
   \d_y f(t)\equiv&\  \d_yf(t,Y_t,\tilde Y_t,\dbE[Y_t],Z_t)
  =\frac{f(t,Y_t,\dbE[Y_t],Z_t)-f(t,\tilde Y_t,\dbE[Y_t],Z_t)}{Y_t-\tilde Y_t},\\
   \d_{\bar y} f(t)\equiv&\  \d_{\bar{y}}f(t,\tilde Y_t,\dbE[Y_t],\dbE[\tilde Y_t],Z_t)
  =\frac{f(t,\tilde Y_t,\dbE[Y_t],Z_t)-f(t,\tilde Y_t,\dbE[\tilde Y_t],Z_t)}{\dbE[Y_t]-\dbE[\tilde Y_t]},\\
  \d_{z} f(t) \equiv&\ \d_{z}f(t,\tilde Y_t,\dbE[\tilde Y_t],Z_t,\tilde Z_t)
  =\frac{f(t,\tilde Y_t,\dbE[\tilde Y_t],Z_t)-f(t,\tilde Y_t,\dbE[\tilde Y_t],\tilde Z_t)}{Z_t-\tilde Z_t},
\end{align*}
and $\d Y$ and $\d Z$ could be defined similarly. Then we have
\begin{align}\label{3.5.1}
  f(t,Y_t,\dbE[Y_t],Z_t)-f(t,\tilde Y_t,\dbE[\tilde Y_t],\tilde Z_t)
  =\d_yf(t)\d Y_t+\d_{\bar y}f(t)\dbE[\d Y_t]+\d_z f(t)\d Z_t.
\end{align}

For one dimensional BSDE \rf{BSDE2}, we have the following comparison theorem.

\begin{theorem}\label{3.6.21}
Let $(Y,Z)$ and $(\tilde Y,\tilde Z)$ be the bounded solutions of BSDE \rf{BSDE2} with generators $(f,\xi)$ and $(\tilde f,\tilde \xi)$ respectively, which satisfy the condition of Lemma \ref{3.4}. Moreover, if
$\xi\les \tilde\xi$ a.s., $f(t,y,\bar y,z)\les\tilde f(t,y,\bar y,z)$ a.e., and the following  conditions hold:
\begin{align*}
\ds   &\hb{{\rm(A1)} for every }Y,\tilde Y,Z, \ \
  \d_y f(t), \d_{\bar y} f(t)  \in L_\dbF^\infty(0,T;\dbR),\\
\ns\ds  &\hb{{\rm(A2)} for every }Z,\tilde Z\in\cZ^2[0,T], \hb{ and any bounded process }Y,\q \d_{z} f(t) \in\cZ^2[0,T],\\
\ns\ds &\hb{{\rm(A3)} one of the two coefficients $f$ and $\tilde f$ is nondecreasing in } \bar y,
\end{align*}
then we have $Y_t\les\tilde Y_t$ a.s. for every $t\in[0,T]$.
\end{theorem}

\begin{remark}
The conditions {\rm (A1)} and {\rm (A2)} hold if there exists a positive constant $C$ such that
\begin{align*}
  |f(t,y_1,\bar y_1,z_1)-f(t,y_2,\bar y_2,z_2)|\les C(|y_1-y_2|+|\bar y_1-\bar y_2|)+C(\tb{1+}|z_1|+|z_2|)|z_1-z_2|
\end{align*}
for all $t\in[0,T]$ and $y_1,\bar y_1,y_2,\bar y_2,z_1,z_2\in\dbR$. Moreover, {\rm (A1)} and {\rm (A2)} hold too if $f(t,y,\bar y,z)$ satisfies the global Lipschitz condition.
\end{remark}

\begin{proof}[Proof of Theorem \ref{3.6.21}]
Without lose of generality, we would like to let that $f$ is nondecreasing in $\bar y$.
Taking the difference of BSDE \rf{BSDE2} with coefficients $(f,\xi)$ and $(\tilde f,\tilde\xi)$ respectively, we obtain that
\begin{align}\label{3.6.2}
  Y_t-\tilde Y_t=&\ Y_0-\tilde Y_0-\int_0^t[f(s,Y_s,\dbE[Y_s],Z_s)-f(s,\tilde Y_s,\dbE[\tilde Y_s],\tilde  Z_s)]ds\\
  &\ -\int_0^t[f(s,\tilde Y_s,\dbE[\tilde Y_s],\tilde Z_s)-\tilde f(s,\tilde Y_s,\dbE[\tilde Y_s],\tilde Z_s)]ds
  +\int_0^t(Z_s-\tilde Z_s)dW_s.
\end{align}
We define the measure $\dbQ$ by $d \dbQ=\mathcal{E}_T(\L)d\dbP$, where
\begin{align*}
   \L_t=\int_0^t \d_{z} f(t) dW_s.
\end{align*}
By Lemma \ref{3.4}, we have $Z,\tilde Z\in\cZ^2[0,T]$. So the conditions  {\rm (A1)} and {\rm (A2)} imply that $\L\in\hb{\rm BMO}$ and hence $\dbQ$ is a probability measure equivalent to $\dbP$.

\ms

We denote by $\bar\L$ the martingale part of $\d Y=Y-\tilde Y$, in other words,
\begin{align*}
 \bar\L_t=\int_0^t(Z_s-\tilde Z_s)dW_s,\q~t\in[0,T].
\end{align*}
Therefore, on the one hand, from \rf{3.5.1}, we have that the process
\begin{align*}
&\d Y_t+\int_0^t[f(s,Y_s,\dbE[Y_s],Z_s)-f(s,\tilde Y_s,\dbE[\tilde Y_s],\tilde  Z_s)]ds\\
&\q\ +\int_0^t[f(s,\tilde Y_s,\dbE[\tilde Y_s],\tilde Z_s)-\tilde f(s,\tilde Y_s,\dbE[\tilde Y_s],\tilde Z_s)]ds\\
  &= \d Y_t+\int_0^t\Big(\d_y f(s)\d Y_s+\d_{\bar y}f(s)\dbE[\d Y_s]+\d_zf(s)\d Z_s \Big)ds
  +\int_0^t\d f(s,\tilde Y_s,\dbE[\tilde Y_s],\tilde Z_s)ds\\
  & = \d Y_t+\int_0^t\Big(\d_y f(s)\d Y_s+\d_{\bar y}f(s)\dbE[\d Y_s]
  +\d f(s,\tilde Y_s,\dbE[\tilde Y_s],\tilde Z_s)\Big)ds
  +\int_0^t\d_zf(s)\d Z_sds.
\end{align*}
On the other hand, from \rf{3.6.2}, the process
\begin{align*}
&\d Y_t+\int_0^t[f(s,Y_s,\dbE[Y_s],Z_s)-f(s,\tilde Y_s,\dbE[\tilde Y_s],\tilde  Z_s)]ds\\
&\q\ +\int_0^t[f(s,\tilde Y_s,\dbE[\tilde Y_s],\tilde Z_s)-\tilde f(s,\tilde Y_s,\dbE[\tilde Y_s],\tilde Z_s)]ds\\
  &=\d Y_0+ \int_0^t(Z_s-\tilde Z_s)dW_s.
\end{align*}
So, by Girsanov's theorem, noting $\d Y_0\in\dbR$, we obtain that the process
\begin{align*}
&\d Y_t+\int_0^t\Big(\d_y f(s)\d Y_s+\d_{\bar y}f(s)\dbE[\d Y_s]
  +\d f(s,\tilde Y_s,\dbE[\tilde Y_s],\tilde Z_s)\Big)ds\\
  &=\d Y_0+ \int_0^t(Z_s-\tilde Z_s)dW_s-\int_0^t\d_zf(s)\d Z_sds\\
  &=\d Y_0+\bar\L_t-\langle\L,\bar\L\rangle_t
\end{align*}
is a local martingale under the measure $\dbQ$. Moreover, Proposition 11 of \cite{Doleans-Meyer79} implies that
\begin{align*}
\bar\L_t-\langle\L,\bar\L\rangle_t\in\hb{BMO}(\dbQ).
\end{align*}
Finally, using the martingale property, the duality principle between SDEs and BSDEs, and the boundary conditions $Y_T=\xi$, $\tilde Y_t=\tilde\xi$, we have that
\begin{align*}
Y_t-\tilde Y_t= \dbE^{\dbQ}_t\bigg(&\ e^{\int_t^T\d_yf(s)ds}(\xi-\tilde\xi)
+\int_t^Te^{\int_t^s\d_{y}f(r)dr}\d_{\bar y}f(s)\dbE[Y_s-\tilde Y_s]ds\\
&+\int_t^Te^{\int_t^s\d_yf(r)dr}\big[f(s,\tilde Y_s,\dbE[\tilde Y_s],\tilde Z_s)
-\tilde f(s,\tilde Y_s,\dbE[\tilde Y_s],\tilde Z_s)\big]ds\bigg),
\end{align*}
which, noting $\xi\les \tilde\xi$ a.s. and $f(t,y,z)\les\tilde f(t,y,z)$ a.e., implies that
\begin{align*}
Y_t-\tilde Y_t\les&\ \dbE^{\dbQ}_t\int_t^Te^{\int_t^s\d_{y}f(r)dr}\d_{\bar y}f(s)\dbE[Y_s-\tilde Y_s]ds\\
=&\ \int_t^T \dbE^{\dbQ}_t\Big(e^{\int_t^s\d_{y}f(r)dr}\d_{\bar y}f(s)\Big)\dbE[Y_s-\tilde Y_s]ds.
\end{align*}
%
%Consider an appropriate approximation of $\phi(y)=y^{+}$, $y\in\dbR$. Using Jensen's inequality and note the inequality $(ab)^{+}\les a\cd b^{+}$ when $a\ges 0$, and $f$ is nondecreasing in $\bar y$, we have
%%
%\begin{align*}
%\phi(\d Y_t)\les&\ \phi\bigg(\frac{1}{T-t}\int_t^T (T-t)\dbE^{\dbQ}_t\Big(e^{\int_t^s\d_{y}f(r)dr}\d_{\bar y}f(s)\Big)\dbE[\d Y_s]ds\bigg)\\
%\les&\ \frac{1}{T-t}\int_t^T \phi\bigg((T-t)\dbE^{\dbQ}_t\Big(e^{\int_t^s\d_{y}f(r)dr}\d_{\bar y}f(s)\Big)\dbE[\d Y_s]\bigg)ds\\
%\les&\ \frac{1}{T-t}\int_t^T (T-t)\Big|\dbE^{\dbQ}_t\Big(e^{\int_t^s\d_{y}f(r)dr}\d_{\bar y}f(s)\Big)\Big|\Big(\dbE[\d Y_s]\Big)^{+}ds\\
%\les&\ \int_t^T \dbE^{\dbQ}_t\Big[\Big|e^{\int_t^s\d_{y}f(r)dr}\d_{\bar y}f(s)\Big|\Big]\dbE[\d Y_s^{+}]ds.
%\end{align*}
%
Since $\d_y f(t)$ and $\d_{\bar y} f(t)$ are bounded processes, there exists a positive constant $K$ such that
$$\Big\|e^{\int_t^s\d_{y}f(r)dr}\d_{\bar y}f(s)\Big\|_{L_\dbF^\infty(0,T)}\les K.$$
Hence,
\begin{align*}
Y_t-\tilde Y_t\les&\ K\int_t^T\dbE[Y_s-\tilde Y_s]ds.
\end{align*}
Notice that $\mathbb{E}[Y_s-\tilde Y_s]\leq \mathbb{E}[(Y_s-\tilde Y_s)^+]$ and then
$$\int_t^T\dbE[Y_s-\tilde Y_s]ds\leq \int_t^T\dbE[(Y_s-\tilde Y_s)^+]ds,$$
which implies that
$$
\Big(\int_t^T\dbE[Y_s-\tilde Y_s]ds\Big)^+\leq \Big(\int_t^T\dbE[(Y_s-\tilde Y_s)^+]ds\Big)^+=\int_t^T\dbE[(Y_s-\tilde Y_s)^+]ds.
$$
Note the inequality $(ab)^{+}\les a\cd b^{+}$ when $a\ges 0$, it follows
\begin{align*}
(Y_t-\tilde Y_t)^+\les&\ K\Big(\int_t^T\dbE[Y_s-\tilde Y_s]ds\Big)^+\leq K\int_t^T\dbE[(Y_s-\tilde Y_s)^+]ds, \q~ t\in[0,T],
\end{align*}
%Therefore we obtain
%
%\begin{align*}
%\dbE[\d Y_t^{+}]\les M\int_t^T \dbE[\d Y_s^{+}]ds, \q~ t\in[0,T],
%\end{align*}
%
from which we can conclude with the help of Gronwall's lemma that $Y_t\les\tilde Y_t$, $t\in[0,T]$,  $\dbP$-a.s.
\end{proof}

\begin{remark}\rm
We point out that the generator $f$ here is of quadratic growth with respect to $z$, which is weaker than that of \cite{Buckdahn2} which is of linear growth with respect to $z$. On the other hand, the terminal value $\xi$ is bounded in our paper, which is stronger than that of the terminal value $\xi$ is in $L^2_{\mathbb{F}}(t,T;\mathbb{R})$ used in Buckdahn, Li and Peng \cite{Buckdahn2}.
\end{remark}

\section{Conclusion}

The multi-dimensional BSDEs with quadratic growth is a difficult yet important topic in the field of BSDEs.
In this work, using an artful method to construct the contracting mapping principle, we proved the existence and uniqueness of  multi-dimensional mean-field BSDEs with quadratic growth under a small terminal value. Moreover, a comparison theorem is obtained. For the general mean-field BSDEs (see Carmona and Delarue \cite{Carmona-Delarue-AP2015,CD18}) with quadratic growth, the relevant results are under our investigation.

%
%\section*{Acknowledgements}
%
%The authors would like to thank the associate editor and the anonymous referees for their insightful comments that improved the quality of this paper.
%

\end{document}